\documentclass[10pt,openright]{article}
\usepackage{amssymb, amsmath}
\usepackage{color}
\usepackage{latexsym}
\usepackage{amscd}
\usepackage{amsmath,amsfonts,amssymb,amsbsy,amstext,amsthm,amscd}
\usepackage{graphicx}

\newtheorem{defi}{Definition}
\newtheorem{theo}{Theorem}
\newtheorem{coro}{Corollary}
\newtheorem{propo}{Proposition}
\newtheorem{lema}{Lemma}
\newtheorem{conj}{Conjecture}

\newcommand{\Lx}[3][]{\mathrm{L}^{#1}_{#2}#3}

\newcommand{\Rx}[3][]{\mathrm{R}^{#1}_{#2}#3}

\newcommand{\gen}[1]{\langle #1 \rangle}

\begin{document}

\title{ Malcev algebras corresponding to smooth almost left automorphic Moufang loops}
\author{ Ramiro Carrillo-Catal\'an \\
\small CONACyT - Universidad Pedag\'ogica Nacional \\[-0.8ex]
\small Unidad 201 Oaxaca\\[-0.8ex]
\small \texttt{rcarrilloca@conacyt.mx},\\
\and  Marina Rasskazova  \\
    \small Omsk Technic University \\[-0.8ex]
\small Omsk,644050, pr.Mira 15\\[-0.8ex]
\small \texttt{marinarasskazova@yandex.ru},\\
\and
Liudmila Sabinina\\
\small Centro de Investigacion en Ciencias   \\[-0.8ex]
\small UAEM, Cuernavaca\\[-0.8ex]
\small \texttt{liudmila@uaem.mx}
\date{June 27, 2017}
}

\maketitle

\begin{abstract}

In this note we introduce the concept of an almost left automorphic Moufang loop and
study the properties of tangent algebras of  smooth  loops of this class. \\

\textbf{Key words:}  \textit{Malcev algebras, Moufang loops, Alternative loops, Automorphic loops, Local almost left automorphic Moufang loops}.\\

\textbf{2010 Mathematics Subject Classification:}  17D10, 20N05.
\end{abstract}

\section{Preliminaries }

In what follows, a
\textit{loop}  structure is a universal algebra $\gen{\mathcal{Q}, \cdot, \backslash, /,1}$ of type $(3,0,1)$
such that the identities $$ (x\cdot y)/ y= x= (x/y)\cdot y,$$
 $$ x\cdot(x\backslash y)= y= x\backslash (x\cdot y),$$
 $$ x\cdot 1=x=1\cdot x.$$ are satisfied   for any two elements in the loop.  For simplicity we will write $xy$ instead of $x\cdot y$.
 A \textit{Moufang loop} is a loop in which \textit{any} of the following equivalent identities
  $$  ((xy)x)z=x(y(xz)); \qquad ((xy)z)y=x(y(zy)); \qquad (xy)(zx)=(x(yz))x.$$
 hold for every three elements of the loop.  Moufang loops are \textit{diassociative}: the subloop generated by every two elements is a group.\\

 For an  element $a$  of $\mathcal{Q}$ the bijections $\Lx{a}:\mathcal{Q}\longmapsto \mathcal{Q}$ and $\Rx{a}: \mathcal{Q}\longmapsto \mathcal{Q}$ given by $\Lx{a}{x}=ax$ and $\Rx{a}{x}=xa$ will be called, respectively, the \textit{left} and the \textit{right translation} by $a$. The left and right translations generate the \textit{multiplication group},
$\mathrm{Mlt}(\mathcal{Q})$.  The stabilizer of the neutral element in  $\mathcal{Q}$ defines the \textit{inner mapping group} denoted by  $\mathrm{Inn}(\mathcal{Q})$. This group is actually a  subgroup of the multiplication group  generated by three different types of elements of $\mathrm{Mlt}(\mathcal{Q})$, namely
$$\ell_{x,y}=\Lx[-1]{xy}\circ\Lx{x}\circ\Lx{y}, \qquad \mathrm{r}_{x,y}=\Rx[-1]{xy}\circ \Rx{y}\circ\Rx{x} \qquad \text{and} \qquad \mathrm{T}_{x}=\Lx[-1]{x}\circ\Rx{x}$$
for every two elements $x$ and $y$ in $\mathcal{Q}$.
If the group  $\mathrm{Inn}(\mathcal{Q})$  acts on $\mathcal{Q}$  by automorphisms, the loop $\mathcal{Q}$ is said to be \textit{automorphic}.  Equivalently, the loop is \textit{automorphic} if the mappings $\ell_{x,y}$, $\mathrm{r}_{x,y}$ and $\mathrm{T}_{x}$ are automorphisms of $\mathcal{Q}$.
\ All the left translations by elements in $\mathcal{Q}$ generate the  \textit{left multiplication} subgroup of $\mathrm{Mlt}(\mathcal{Q})$, denoted by $\mathrm{LMlt}(\mathcal{Q})$.  The stabilizer of the neutral element in   $\mathrm{LMlt}(\mathcal{Q})$ defines the \textit{left inner mapping} group, $\mathrm{LInn}(\mathcal{Q})$, generated by  the mappings $\ell_{x,y}.$ If  $\mathrm{LInn}(\mathcal{Q})$ is a group of  automorphisms of $\mathcal{Q}$

then $\mathcal{Q}$ will be called a \textit{left automorphic} loop.\\

\begin {defi} A Moufang loop  $L$  with the property that every three elements of $L$  generate a left  automorphic subloop
 will be called an \textit{ almost  left automorphic} Moufang loop.\\
\end{defi}
In the differential geometry framework the notion of a loop being  left automorphic is important: a loop with the  left automorphic and left power alternativity properties  completely defines a reductive homogeneous space. [K3], [SLV].
Following  ideas of M. Kikkawa  [K1] we  have studied [CS1] the so-called \textit{reductive Kikkawa} spaces. There it is shown that smooth left automorphic power alternative loops are Moufang.

 The corresponding tangent algebra  $\mathcal{A}$ of a reductive  Kikkawa space is characterized as an anticommutative algebra satisfying the relations
\begin{align*}
J(xy, z, u)&+ J(yz, x, u) + J(zx, y, u) = 0, \\
  J(x, y, uv)&=  J(x, y, u)v - J(x, y, v)u , \qquad  \forall  x, y, z, u, v \in \mathcal{A}.
\end{align*}
where $J(x, y, z):=(xy)z+(yz)x+(zx)y$. \\

In ([CS1], p.6, T.2,13 ) it was also proved that tangent algebras of reductive Kikkawa spaces are a certain type of \textit{Malcev} algebras (anticommutative algebras satisfying  $J(x, y, xz)=J(x, y, z)x$, see [Sa]).  More precisely, we have
\newpage
\begin{theo}\label{CS} [Carrillo-Catal\'an, Sabinina]
Let  $\mathcal{A}$ be an anticommutative algebra  for which the following identities hold for every
five elements $x, y, z, u, v$ of $\mathcal{A}$:
\begin{align}
J(xy, z, u)&+ J(yz, x, u) + J(zx, y, u) = 0, \label{firstp} \\
  J(x, y, uv)&=  J(x, y, u) v - J(x, y, v) u, \label{secondp}
\end{align}
If the characteristic of $\mathcal{A}$ is not 2 then $\mathcal{A}$  is contained in the variety of  Malcev algebras defined by the identity:
\begin{align}
J(x, y, z)x&=J(x, y, xz)=0 \label{trivialmalcev}, \qquad  \forall  x, y, z \in \mathcal{A}.
\end{align}
\end{theo}

It turns out that the converse of last theorem is not true. In what follows, let us call \textit{Malcev algebras of the first type} those Malcev algebras which satisfy the identities (\ref{firstp}) and (\ref{secondp}). In the same way,  those  Malcev algebras which satisfy (\ref{trivialmalcev}) will be called \textit{ Malcev algebras of the second type}.\\

Malcev algebras of the second type form a larger variety of algebras than Malcev algebras of the first type.
In 2015 I. Shestakov  constructed an algebra  of dimension $29$ by using the Non Associative Algebra System [Albert] (Jacobs, Muddana and Offutt, 1993),
which is a Malcev algebra of the second type, but  not a Malcev algebra of the first type, [SIP].\\
In this paper we construct a new example of Malcev algebra of the second type, but  not a Malcev algebra of the first type of dimension
$23$, moreover, we think that this example has the minimal dimension.
\begin{conj}\label{c1}
Every Malcev algebra the second type of dimension $<23$ over a field of characteristic 0 is an algebra of the first type.
\end{conj}

In this work we try  to answer the question:
What kind of smooth loop corresponds to a Malcev algebra of the second type?

\section{Malcev algebras of the first type}

  It is possible to get equivalent characterizations of Malcev algebras of the first type by using Sagle's properties of Malcev algebras.
The following result gives the equivalences. \\

\begin{lema} $\mathcal{A}$ is
an anticommutative algebra  for which the  identities (\ref{firstp})  and  (\ref{secondp}) hold for every
five elements of $\mathcal{A}$ if and only if  $\mathcal{A}$  is a Malcev algebra such that for every four elements of $\mathcal{A}$ the identity
\begin{align}
J(x, y, uv)&=0  \label{zerojacobian}
\end{align}
holds, provided the characteristic of $\mathcal{A}$ is not 2.  In the special case where the characteristic of $\mathcal{A}$ is neither 2 nor 3, the former  identity  is equivalent to having
\begin{align}
J(x, y ,z)u&=0,   \label{zerojacobian2}
\end{align}
for every four elements of $\mathcal{A}$.
 \end{lema}

\begin{proof}
From the last result we know that $\mathcal{A}$ is a Malcev algebra for which (\ref{trivialmalcev}) holds. Now, since the characteristic of $\mathcal{A}$  is not 2, the identity (\ref{firstp}) can be expressed, by using Sagle's properties for Malcev algebras ([Sa], p.429, L. 2.10, 2.14) and some odd permutations,  as:
\begin{align*}
-2uJ(x, y, z)&=-J(u, x, yz)-J(u, y, zx)-J(u, z, xy)\\
&=J(yz, x, u)+J(zx, y, u)+J(xy, z, u)\\
&=0.
\end{align*}
Therefore the identity (\ref{firstp}) can be replaced by the expression
$J(x, y, z)u=0$.
   Moreover, this identity can be used to rewrite (\ref{secondp}). Namely, If  $J(x, y, z)u=0$ holds  for any four elements of $\mathcal{A}$ then
  $J(x, y, u)v=J(x, y, v)u=0$  and therefore $J(x, y, uv)=0$. \\

Conversely, suppose $\mathcal{A}$  is a Malcev algebra such that  for every four elements (\ref{zerojacobian}) holds.
Then, being a Malcev algebra, the linearized form of the Malcev property  ([Sa], p.429, 2.7):
$$J(x, y, wz)=J(x, y, z)w+J(w, y, z)x-J(w, y, xz) $$ is satisfied for any  $x, y, z, w$ in $\mathcal{A}$.  In particular, if (\ref{zerojacobian}) holds, then
$$J(x, y, v)u=-J(u, y, v)x$$
for any four elements in $\mathcal{A}$. Then, after an odd permutation,  for any four elements in the algebra:
$$J(x, y, u)v-J(x, y, v)u=-J(v, y, u)x-(-J(u, y, v))x=2J(u, y, v)x. $$

If $\mathcal{A}$ is a Malcev algebra with characteristic different from 2, then $2J(u, y, v)x=0$ ([Sa], p.429, 2.14). Hence the identity (\ref{secondp}) holds trivially:
$$ J(x, y ,uv)-J(x, y, u)v+J(x, y, v)u=0.$$

For a   Malcev algebra of  characteristic different from 2 or 3  the identities $J(x, y, uv)=0$ and $J(x, y, z)u=0$ are equivalent. 
To see this, suppose  that $J(x, y, uv)=0$ for every four elements of the algebra. Since the characteristic of the algebra is not 2 and due to Malcev algebra properties ([Sa], p.429, 2.14) we should have:
$$-2uJ(x, y, z)=J(yz, x, u)+J(zx, y, u)+J(xy, z, u)=0, $$
hence  $J(x, y, z)u=0$. Conversely, assume the characteristic of the algebra is different from 3, then once again due to Malcev algebra properties 
([Sa], p.430, 2.15) we know that
$$3J(wx, y, z)=J(x, y, z)w-J(y, z, w)x-2J(z, w, x)y+2J(w, x, y)z.$$

Therefore, if  $J(x,y,z)u=0$ then $J(wx,y,z)=0$ for any four elements of the algebra.
Hence $J(y,z,wx)=0$.\\

Finally, to prove (\ref{zerojacobian2}), suppose that for every four elements in $\mathcal{A}$ the  identities (\ref{firstp})  and  (\ref{secondp}) hold. Then by Theorem \ref{CS} $\mathcal{A}$ is a Malcev algebra. As before, the identity (\ref{firstp}) and Sagle's properties for Malcev algebras  ([Sa], p.429, 2.14) can be combined to get
(\ref{zerojacobian2}).  Conversely, we know that (\ref{zerojacobian2}) holds in $\mathcal{A}$  if and only if (\ref{zerojacobian}) is satisfied for every four elements of the algebra.  But the identity (\ref{zerojacobian}) directly implies (\ref{firstp})  and  (\ref{secondp}) as shown above.\\
\end{proof}

From all of the above, we conclude that an anticommutative algebra with characteristic different from 2 and 3 and satisfying the identities (\ref{firstp}) and  (\ref{secondp}) is equivalent to a Malcev algebra for which  (\ref{zerojacobian}) or   (\ref{zerojacobian2}) holds for every four elements in the algebra. In other words,
Malcev algebras of the first type can be characterized as algebras for which (\ref{firstp}) and  (\ref{secondp}) hold or as Malcev algebras satisfying (\ref{zerojacobian}) or as Malcev algebras satisfying (\ref{zerojacobian2})
or as anticommutative algebra satisfying (\ref{zerojacobian}) and (\ref{zerojacobian2}),
provided the characteristic of the algebra is not 2 nor 3.\\

In what follows we will use  the definition of Malcev algebra of the first type  as a  Malcev algebra with the identity  (\ref{zerojacobian}). It is known that the tangent algebra of a smooth diassociative loop is a Binary Lie algebra, which can be defined by the identity $J(x,y,xy)=0$
(see  [GA])

\section{Malcev algebras of second type}

Let $\mathsf{A}$ be a Malcev algebras of the second type. In this section we will prove that such algebras turn out to be tangent of a local  almost left automorphic Moufang loops. \\

We begin with  statements  which are direct consequences of Malcev algebras properties:
\begin{propo}
If $y$ and $z$ are two elements of $\mathsf{A}$ such that $J(y,z,\mathsf{A})=0$ then $yz$ belongs to the Lie kernel of $\mathsf{A}$, that is, $yz\in N(\mathsf{A})=\{x\in \mathsf{A}\  | \ J(x,\mathsf{A},\mathsf{A})=0\}$.
\end{propo}
\begin{proof}
Since $\mathsf{A}$ is a  Malcev  algebra, it follows from Sagle ([Sa], p.431, 2.26) that the identity \begin{align}
J(wx,y,z)=wJ(x,y,z)+J(w,y,z)x+2J(yz,w,x), \label{sagle1}
\end{align}
holds for any four elements of the algebra. Now, if $J(y,z,\mathsf{A})=0$  then $2J(yz,\mathsf{A},\mathsf{A})=0$.\\
\end{proof}

\begin{propo}

$\mathsf{A}^{4}\subseteq N(\mathsf{A})$.
\end{propo}
\begin{proof}
By the last result, since the identity  (\ref{trivialmalcev}) holds in $\mathsf{A}$, we know that  $x(xz)\in N(\mathsf{A})$ for any two elements of the algebra.  Then the quotient algebra $\mathsf{A}/N(\mathsf{A})$ must satisfy  $x(xy)=0$.
As proved in ([EA], Prop 3.1, p.1-5)  we can conclude that $\left(\mathsf{A}/N(\mathsf{A})\right)^{4}=0$ or, equivalently,
$\mathsf{A}^{4} \subseteq N(\mathsf{A})$.\\
\end{proof}

The above ideas can be generalized by using the fact that the Jacobian is a skew-symmetric multilinear map.
\begin{lema}
The multilinear maps $\xi, \zeta, \varsigma: \mathsf{A}\times\mathsf{A}\times\mathsf{A}\times\mathsf{A}\longrightarrow\mathsf{A}$  given by
\begin{align*}
\xi (x_{1},x_{2},x_{3},x_{4})&= J(x_{1},x_{2},x_{3}x_{4}) \\
\zeta (x_{1},x_{2},x_{3},x_{4})&=  J(x_{1},x_{2},x_{3})x_{4} \\
\varsigma(x_{1},x_{2},x_{3},x_{4},x_{5})&= J(x_{1}x_{2},x_{3}x_{4},x_{5})
\end{align*}
are skew-symmetric.
\end{lema}
\begin{proof}
Since $\mathsf{A}$ is anticommutative and due to the properties of the Jacobian it is clear that $\xi (x_{1},x_{2},x_{3},x_{4})= J(x_{1},x_{2},x_{3}x_{4})$ is skew-symmetric in $x_{1}$ and $x_{2}$ as well as in $x_{3}$ and  $x_{4}$.   If  (\ref{trivialmalcev}) holds for $\mathsf{A}$, the map $\xi$ is also skew-symmetric for   $x_{1}$ and $x_{3}$ since  $J((x_{1}+x_{3}),x_{2},(x_{1}+x_{3})x_{4})=0$  can be rewritten  by linearity as
\begin{align*}
 J(x_{1},x_{2},x_{1}x_{4})+J(x_{3},x_{2},x_{3}x_{4})+J(x_{1},x_{2},x_{3}x_{4})+J(x_{3},x_{2},x_{3}x_{4})=0.
\end{align*}
Using (\ref{trivialmalcev}) we obtain:
\begin{align*}
J(x_{1},x_{2},x_{3}x_{4})=-J(x_{3},x_{2},x_{1}x_{4}).
\end{align*}
 Notice that if $\xi$ is skew-symmetric in $x_1$ and $x_3$ it is also skew-symmetric in $x_1$ and $x_4$:
 $$J(x_{1},x_{2},x_{3}x_{4})=-J(x_{1},x_{2},x_{4}x_{3})=J(x_{4},x_{2},x_{3}x_{1})=-J(x_{4},x_{2},x_{1}x_{3}).$$

  The proof that the map $\zeta$ is skew-symmetric is similar: due to the properties of the Jacobian $\zeta$ is skew-symmetric in $x_1$ and $x_2$ as well as in $x_1$ and $x_3$.   Now by linearity and applying (\ref{trivialmalcev}),
  the fact that $J((x_{1}+x_{4}),x_{2},x_{3})(x_{1}+x_{4})=0$ implies the skew-symmetry in $x_1$ and $x_4$.\\

As above, $\varsigma$ is skew-symmetric in  $x_1$ and  $x_2$ as well as in $x_1$ and $x_4$.  The skew-symmetry of $\xi$ can be used to prove that $\varsigma$ is skew-symmetric in  $x_{i}$ and $x_{5}$ for $i=1,2,3,4$. Notice that
$J(x_{1}x_{2},x_{3}x_{4}, x_{5})=J(x_{5}, x_{1}x_{2},x_{3}x_{4})$ due to an even permutation.  Since $\xi$ is skew-symmetric in $x_1$ and $x_3$, $\varsigma$ is skew-symmetric in $x_5$ and $x_3$: $J(x_{5}, x_{1}x_{2},x_{3}x_{4})=-J(x_{3}, x_{1}x_{2},x_{5}x_{4})$,  as well as in  $x_5$ and $x_4$: $J(x_{5}, x_{1}x_{2},x_{3}x_{4})=-J(x_{4}, x_{1}x_{2},x_{3}x_{5})$.  To  verify that $\varsigma$ is skew-symmetric in $x_5$ and $x_2$ we also use the skew-symmetry of $\xi$, suitable permutations and the properties of the Jacobian:
$$J(x_{5}, x_{1}x_{2},x_{3}x_{4})=-J(x_{5}, x_{3}x_{4},x_{1}x_{2})=-J(x_{1}, x_{5}x_{2},x_{3}x_{4}),$$
and similarly for $x_5$ and $x_1$. The result follows since the permutations $(i,5)$ for $i=1,2,3,4$ generate the symmetric group in $\{1,2,3,4,5\}$.

\end{proof}

\begin{theo}\label{th}
Let  $\mathsf{A}$ be an anticommutative algebra satisfying (\ref{trivialmalcev}). Then
\begin{enumerate}
\item[A.] $\mathsf{A}^{3}$ belongs to the Lie Kernel $N(\mathsf{A})$. In particular $\mathsf{A}/N(\mathsf{A})$ is a nilpotent Lie algebra.
\item[B.] $\mathsf{A}^{2}$ is a Lie algebra.
\item[C.] $\mathsf{A}$ satisfies the identity:
\begin{align*}
2wJ(x,y,z)=3J(w,x,yz).
\end{align*}
\item[D.] The following hold:
\begin{enumerate}
\item[i.]  $J(\mathsf{A}^{i},\mathsf{A}^{j},\mathsf{A}^{k})=0 \qquad \ \quad\text{if} \qquad i+j+k\geq5$.
\item[ii.] $J(\mathsf{A}^{i},\mathsf{A}^{j},\mathsf{A}^{k})\mathsf{A}^{r}=0 \qquad \text{if} \qquad i+j+k+r\geq5$.
\item[iii.]  $(J(\mathsf{A},\mathsf{A},\mathsf{A})\mathsf{A})\mathsf{A}=0$.
\end{enumerate}
\item[E.] $J(\mathsf{A},\mathsf{A},\mathsf{A})^{2}=0$.\\
\end{enumerate}
\end{theo}
\begin{proof}

Notice that due to anticommutativity  we know  that  $$J(x_{1}x_{2},x_{3}x_{4},x_{5})=-J(x_{3}x_{4},x_{1}x_{2},x_{5})$$ for any elements $x_{1},x_{2},x_{3},x_{4},x_{5}$ of $\mathsf{A}$.
In the previous Lemma, the map $\varsigma$ was proved to be skew-symmetric. Then we can write

\begin{align*}
J(x_{1}x_{2},x_{3}x_{4},x_{5})=-J(x_{3}x_{2},x_{1}x_{4},x_{5})=J(x_{3}x_{4},x_{1}x_{2},x_{5}).
\end{align*}
We conclude that $J(\mathsf{A}^{2},\mathsf{A}^2,\mathsf{A})=0$.  Now, due to the skew-symmetry of the map $\xi$, for every  $J(xy,ab,c)\in J(\mathsf{A}^{2},\mathsf{A}^2,\mathsf{A})$  we have

$$J(xy,ab,c)=-J(x(ab),y,c).$$

In the same way, by using the skew-symmetry of $\varsigma$  followed by the skew-symmetry of the map $\xi$ we get:

$$J(xy,ab,c)=-J(xy,ac,b)=J(x(ac),y,b).$$

Hence
\begin{align*}
0=J(\mathsf{A}^{2},\mathsf{A}^2,\mathsf{A})=J(\mathsf{A}^{3},\mathsf{A},\mathsf{A}),
\end{align*}
which implies that $\mathsf{A}^{3}\subseteq N(\mathsf{A})$. On the other hand, $\mathsf{A}/N(\mathsf{A})$ is a Lie algebra since $J(\mathsf{A},\mathsf{A},\mathsf{A})\subset\mathsf{A}^{3}\subseteq N(\mathsf{A})$ and  $N(\mathsf{A})$
is zero in $\mathsf{A}/N(\mathsf{A})$. This proves (\textit{A}), (\textit{B}) and part (\textit{i})  of (\textit{D}) of the Theorem.\\

On the other hand, since $\mathsf{A}$ is a Malcev algebra, once again according to Sagle's properties for Malcev algebras ([Sa], p.429, L. 2.10, 2.14), for every $x,y,z,w$ of $\mathsf{A}$ we should have
$$2wJ(x,y,z)=J(w,x,yz)+J(w,y,zx)+J(w,z,xy).$$

It is enough to use this identity and the skew-symmetry of the map $\xi$ to prove (\textit{C}): each of the terms can be written as $J(w,y,zx)=-J(w,y,xz)=J(w,x,yz)$ and $J(w,z,xy)=-J(w,x,zy)=J(w,x,yz)$. \\

The last part of  (\textit{A}) follows from (\textit{C}) and (\textit{i}): from (\textit{C}) we know that $2\mathsf{A}J(\mathsf{A},\mathsf{A},\mathsf{A})=3J(\mathsf{A},\mathsf{A},\mathsf{A}^{2})$.  This implies  that
 $2\mathsf{A}J(\mathsf{A},\mathsf{A},\mathsf{A}^{2})=3J(\mathsf{A},\mathsf{A},\mathsf{A}^{3})=0$
 due to (\textit{i}).  We conclude that $\mathsf{A}J(\mathsf{A},\mathsf{A},\mathsf{A}^{2})=0$. Therefore
$$(2\mathsf{A}J(\mathsf{A},\mathsf{A},\mathsf{A}))\mathsf{A}=\left(3J(\mathsf{A},\mathsf{A},\mathsf{A}^2)\right)\mathsf{A}=3J(\mathsf{A},\mathsf{A},\mathsf{A}^2)\mathsf{A}=0.$$
Since $(2\mathsf{A}J(\mathsf{A},\mathsf{A},\mathsf{A}))\mathsf{A}=-2(J(\mathsf{A},\mathsf{A},\mathsf{A})\mathsf{A})\mathsf{A}$ we conclude that $(J(\mathsf{A},\mathsf{A},\mathsf{A})\mathsf{A})\mathsf{A}=0$.\\

 Finally  (\textit{C}) implies that  $J(\mathsf{A},\mathsf{A},\mathsf{A})\mathsf{A}^{3}=0$.  Since  $J(\mathsf{A},\mathsf{A},\mathsf{A})\subseteq\mathsf{A}^{3}$, we get:

\begin{align*}
 J(\mathsf{A},\mathsf{A},\mathsf{A})^{2}\subseteq J(\mathsf{A},\mathsf{A},\mathsf{A})\mathsf{A}^{3}=0.
\end{align*}
\end{proof}

\begin{coro}
If $\mathsf{A}$ is a  semiprime algebra, then $\mathsf{A}$ is a Lie algebra.
\end{coro}
\begin{proof}

Suppose that $\mathsf{A}$ is a semiprime algebra.  Since, by the previous Theorem, the ideal $J(\mathsf{A},\mathsf{A},\mathsf{A})$  satisfies  $J(\mathsf{A},\mathsf{A},\mathsf{A})^2=0$ it follows that $J(\mathsf{A},\mathsf{A},\mathsf{A})=0$, namely, $\mathsf{A}$ is a Lie algebra.
\end{proof}

\begin{theo}

A $3 $-generated Malcev algebra  $\mathsf{A}$ of the second type  is a Malcev algebra of the first type.

\end{theo}
\begin{proof}
Let $\mathsf{A}$ be a Malcev algebra of the second type generated by the elements $a,b,c.$
Let $\omega$ be an arbitrary element of $\mathsf{A}$. Since $\mathsf{A}$ is a $3$- generated algebra, $\omega$
is a formal linear combination of products of the generators $a,b,c$.

Let us show that the  identity (\ref{zerojacobian}) holds in $\mathsf{A}$. First, consider  $J(x,y,c\omega)$, where $c$ is an arbitrary generator of $\mathsf{A}$. Since the Jacobian is a multilinear operator, $J(x,y,c\omega)$ can be

simplified using   (\textit{4},\,\textit{i}) of Theorem \ref{th}:
\begin{align*}
J(x,y,c\omega)=\ell_{1}J&(a,b,ca)+\ell_{2}J(a,b,cb)+\ell_{4}J(a,b,c(ab))+
\ell_{5}J(a,b,c(ac))+\\&\ell_{6}J(a,b,c(bc))+\ell_{7}J\left(a,b,c((ab)c)\right)+\ell_{8}J\left(a,b,c(a(bc))\right)+\cdots
\end{align*}

Using again  the fact that $\mathsf{A}^{3} \subseteq N(\mathsf{A})$  and  $J(a,b,cb)=0$, we conclude that

$J(x,y,c\omega)=0$ for any $\omega$ in  $\mathsf{A}$.\\

The general case, $J(x,y,uv)=0$  where $x,y,u$ and $v$ are arbitrary elements of $\mathsf{A}$,
 can be handled by analogous considerations.

\end{proof}

Due to Malcev -Kuzmin theory [Kuz], we have

\begin{coro}
A Malcev  algebra of the second type  is, in fact, a tangent algebra of a  local almost left automorphic Moufang loop.
The tangent algebra of every smooth almost left  automorphic Moufang loop is  a Malcev algebra of the second type.
\end{coro}

{\bf Remark} The question of the existence of global  almost left automorphic Moufang loop, which corresponds to the the given 
Malcev algebra of the second type is solved positevly  in the article [CGRS].

\section{Example}

 In this section we discuss an example of a $23$-dimensional algebras of second type which is not an algebra of first type. \\

 Let $F$ be a free anti commutative algebra generated by $X=\{x_1, x_2, x_3, x_4\}$, nilpotent of class $4$, it means $F^{4}=0$. Let $I$ be a subspace with  a basis of all $X$-words, $w=w(x_1, x_2, x_3, x_4)$, such that some letter $x_i$ appears in $w$  two or more times. It is clear that $I$ is an ideal of $F$.  Let's denote by $\mathrm{A}$ the factor algebra $F/I$. Then a basis of $\mathrm{A}$ has 22 elements:  $B=\cup ^{3}_{i=1}B_i$ with 

 \begin{align*}
 B_1=&X,\\
 B_2=&\{ [x_i,x_j] \ | 1\leq i < j \leq 4  \},\\
 B_3=&\{ [x_i,x_j,x_k] \ | 1\leq i < j \leq 4, 1\leq k\leq 4, \ k\neq i, j  \}.
 \end{align*}

 The algebra $\mathrm{A}$ is a Malcev algebra.\\

 Let's define an antisymmetric bilinear function $\psi:\mathrm{A}\times \mathrm{A}\longmapsto k$  given by the following values:\\

 $$
  \begin{array}{ll}
 \psi([x_1,x_2], [x_3,x_4])=2, &\psi([x_1,x_3], [x_2,x_4])=-2,\\
  \psi([x_1,x_4], [x_2,x_3])=2,&\psi([x_2,x_3,x_1], x_4)=-3,\\
 \psi([x_2,x_4,x_1], x_3)=3,&\psi([x_2,x_4,x_3], x_1)=-1,\\
 \psi([x_3,x_4,x_1], x_2)=-3,& \psi([x_3,x_4,x_2], x_1)=1,\\
 \end{array}
 $$

 and $\psi(v,w)=0$ for all other values.\\

 Consider a space $\tilde{\mathrm{A}}=\mathrm{A}\oplus kv$ and define a product on $\tilde{\mathrm{A}}$ as follows:
 \begin{align}
[(a,\alpha v),(b, \beta v)]=([a,b], \psi(a,b)) \label{exampleproduct}
 \end{align}

Direct computations show that $\tilde{\mathrm{A}}$ is a second type Malcev algebra. On the other hand, if we set $x_i=(x_i, 0)$, then
\begin{align*}
[ J(x_1,x_2,x_3),x_4 ]&=[x_1,x_2,x_3,x_4]+[x_2,x_3,x_1,x_4]-[x_1,x_3,x_2,x_4]\\
&=(0, -3v)\neq 0.
\end{align*}

Hence $\tilde{\mathrm{A}}$ is not an algebra of the first type.

\section*{Acknowledgments}
The authors thank  Alberto Elduque, Alexander Grishkov and Ivan Shestakov for useful comments.

\section*{Funding}

The first author thanks  CONACYT and Universidad Pedag\'ogica Nacional Unidad 201 Oaxaca for supporting the C\'atedras CONACYT project  1522.   The second author thanks CNPq (Brasil), grant 308221/2012-5 and the third author thanks SNI and FAPESP grant process 2015/07245-4 for support.


\newpage
\begin{thebibliography}{GGW2}



\bibitem [CS1]{CS1} Carrillo-Catal\'an R., Sabinina  L.
\emph{On smooth power-alternative loops}, Communications in algebra, Vol. 32, No. 8, pp. 2969 - 2976, (2004).
\bibitem [CPS]{CPS}  Chein O.,  Pflugfelder H.O. and  Smith J.D.H.,
\emph{Quasigroups and loops: Theory and applications}, Sigma
Series in Pure Mathematics, \textbf{8}, Heldermann Verlag, (1990).
\bibitem [EA]{EA} Elduque A.,
\emph{Quadratic Alternative Algebras}, J. Math. Phys.31,  1, 1-5, 17A45 (17D05), (1990).
\bibitem [Ga]{Ga} Gainov  A.T.\, {\em Binary Lie algebras of characteristic two.}  Algebra and Logic , 8 : 5 (1969) pp. 287 - 297 Algebra i Logika , 8 : 5 pp. 505 - 522, (1969).
\bibitem [CGRS]{CGRS} Grishkov A.,Carrillo Catalan R.,Rasskazova M.,Sabinina L. {\em Nilpotent by Lie center Malcev algebras and corresponding  analytic Moufang loops.} Preprint.
\bibitem [Ker]{Ker} Kerdman  F.S. \emph{Analytic Moufang loops in the large}, Algebra and Logic (1979) 18: 325. Translated from Algebra i Logika, Vol. 18, No. 5, pp. 523 - 555, September -October, (1979).
\bibitem [K1]{K1} Kikkawa M.,
\emph{On local loops in affine manifolds}, J.Sci. Hiroshima Univ. Ser. A-I, 28 , 199 - 207. (1964).
\bibitem [K3]{K3} Kikkawa  M.,
\emph{Geometry of homogeneous Lie loops}, Hiroshima Math. J.  5:141-179, (1975).
\bibitem [Kuz] {Kuz} Kuz'min, E.N. \, {\em  On the relation between Mal'tsev algebras and analytic Mufang groups.} \   Algebra and  Logic (1971) 10: 1. 
Translated from Algebra i Logika, Vol. 10, No. 1, pp. 3 -22, January -February, (1971).
\bibitem [SLV] {} Sabinin Lev V. {\em Smooth Quasigroups and Loops} \ Kluver Academic Publishers. Dordrecht/Boston/London\ 1999.
\bibitem [Sa] {}  Sagle  A.\, {\em Malcev Algebras.} \ Trans. Amer. Math. Soc.,101 (1961), 3, 426 - 458 MR 26 1343.(1963).
\bibitem [SIP]{She} Shestakov  I.P. \, {\em Private Communication.}


\end{thebibliography}
\end{document}